\documentclass[twoside,leqno,12pt]{article}

\usepackage{amssymb}
\usepackage{amsmath}
\usepackage{amscd}
\usepackage{ifthen}
\usepackage{amsthm}
\usepackage{calc}
\usepackage{graphicx}
\usepackage[all]{xy}
\usepackage{color}
\usepackage{multirow}
\usepackage{ytableau}
\usepackage{enumerate}

\newenvironment{purple}{\color{magenta}}{}

\definecolor{darkgreen}{rgb}{0,0.5,0}

\numberwithin{equation}{section}
\newtheorem{thm}[equation]{\sc Theorem}
\newtheorem{lem}[equation]{\sc Lemma}

\newtheorem{prop}[equation]{\sc Proposition}
\newtheorem{alg}[equation]{\sc Algorithm}
\newtheorem{rem}[equation]{\sc Remark}

\newtheorem{defin}{\it Definition}
\newtheorem{ex}{\it Example}

\makeatletter
\renewcommand{\@seccntformat }[1]{\csname the#1\endcsname. }
\makeatother

\chardef\xnearrow='045
\chardef\xnwarrow='055
\chardef\xsearrow='046
\chardef\xswarrow='056

 \newcommand\Hom{\mbox{\rm Hom}}

 \newcommand\Ker{\mbox{\rm Ker}}
 \newcommand\Coker{\mbox{\rm Coker}}
  \newcommand\Ext{\mbox{\rm Ext}}

\newcommand\homleq{\leq_{\rm hom}}

\newcommand\degleq{\leq_{\rm deg}}

\begin{document}
\thispagestyle{empty}
\color{black}
\phantom m\vspace{-2cm}

\bigskip\bigskip
\begin{center}
{\large\bf Combinatorial algorithms for binary operations on LR-tableaux with entries equal to 1 with applications to nilpotent linear operators
}
\end{center}

\smallskip

\begin{center}
  Mariusz Kaniecki and Justyna  Kosakowska
                  \\
                  \vspace{0.5cm}
   Faculty of Mathematics and Computer Science,\\
    Nicolaus Copernicus University in Toru\'n\\
    Chopina 12/18, 87-100 Toru\'n, Poland \\
   kanies@mat.umk.pl; justus@mat.umk.pl

\end{center}



\begin{abstract}
In the paper we investigate an~algorithmic associative binary operation $*$ on the set
$\mathcal{LR}_1$ of Littlewood-Richardson tableaux with entries equal to one. We extend $*$ to an~algorithmic nonassociative binary operation on the set 
$\mathcal{LR}_1\times \mathbb{N}$ and show
that it is equivalent to the operation of taking the generic extensions of objects in the category
of homomorphisms from semisimple nilpotent linear operators to   nilpotent linear operators.  Thus we get a~combinatorial algorithm computing generic extensions in this category.
\end{abstract}

\vspace{-0.3cm}

\section{Introduction}

The main aim of the paper is to generalize results presented in \cite{kk-generic}, where
an~associative operation $*$ on the set
$\mathcal{LR}_1$ of Littlewood-Richardson tableaux (LR-tableaux) with entries equal to one is defined. 
It is proved  that $*$ is associative and it is equivalent to the operation of taking the generic extension of semisimple invariant subspaces of nilpotent linear operators, see \cite[Lemma 5.12]{kk-generic}.

We extend the operation $*$ to a~nonassociative operation on the set $\mathcal{LR}_1\times \mathbb{N}$ and investigate its properties. In particular, in Theorem \ref{thm-main-gen},  we prove that this operation is equivalent to the operation of taking the generic extensions of objects in the category $\mathcal{H}_1$
of homomorphisms from semisimple nilpotent linear operators to   nilpotent linear operators. In particular this gives a~combinatorial algorithm that applying operations on LR-tableaux computes generic extensions in $\mathcal{H}_1$.

  We are motivated by results presented in \cite{kos-sch,kos-sch2}, where there are investigated  relationships 
  between Littlewood-Richardson tableaux and  geometric properties of invariant subspaces of nilpotent linear operators.
 It is proved there that these relationships  are deep and interesting.  On the other hand, in \cite{reineke}, the existence of generic extensions for 
  Dynkin quivers is proved and their connections  with Hall algebras are investigated. 
 Moreover, by results presented in
\cite{bongartz,dengdu,dengdumah,reineke} generic extensions of nilpotent linear operators 
exist and the operation of taking the generic
extension provides the set of all isomorphism classes of
nilpotent linear operators with a~monoid structure. There are many results 
concerning this monoid and its properties (see \cite{dengdu,dengdumah,hubery,kos,reineke}).

The paper is organized as follows.
\begin{itemize}
    \item In Section \ref{sec-def-alg} we define algorithmically  a~binary operation $*$
    on the set $\mathcal{LR}_1\times \mathbb{N}$ and illustrate it by examples. Moreover we present an example showing that $*$ is non-associative.
    \item Section \ref{sec-invariant} contains basic definitions and facts concerning the category of homomorphisms between nilpotent linear operators. We also recall there definitions of generic extensions and hom-order.
    \item In Section  \ref{sec-generic} we prove that the operation $*$ is equivalent to the operation of taking generic extensions in some subcategory of the category of homomorphisms between nilpotent linear operators (Theorem \ref{thm-main-gen}). 
\end{itemize}

\section{Operations on $\mathcal{LR}_1$ and on $\mathcal{LR}_1\times \mathbb{N}$} \label{sec-def-alg}

The aim of this section is to present algorithmic  definitions of combinatorial binary operations on the sets 
$\mathcal{LR}_1$ and $\mathcal{LR}_1\times \mathbb{N}$. Later these operations are applied to computing generic extensions in some category of homomorphisms of nilpotent linear operators. 

Let $\alpha = (\alpha_1, \ldots , \alpha_n,\ldots)$ be a partition (i.e. a sequence of non-negative
integers containing only finitely many non-zero terms and such that $\alpha_1\geq\alpha_2\geq \ldots \geq \alpha_n\geq\ldots).$ Denote by $\overline{\alpha} = (\overline{\alpha}_1,\ldots ,\overline{\alpha}_n,\ldots)$ the dual partition of $\alpha$, i.e.
$\overline{\alpha}_i = \#\{j; \alpha_j \geq i\},$
where $\#X$ denotes the cardinality of a~finite set $X$.
Moreover let $|\alpha|=\alpha_1+\alpha_2+\ldots$. Given partitions $\alpha$ ,$\beta$ we denote
by $\alpha\cup \beta$ the union of these partitions,
i.e. the multiset of parts of  $\alpha\cup \beta$ is the union of multisets of parts of $\alpha$ and $\beta$.

Let $(\alpha,\beta,\gamma)$ be a~partition triple such that $\alpha_1\leq 1$. An~{\bf LR-tableau} of type $(\alpha,\beta,\gamma)$ is  a~skew diagram of shape $\beta\setminus \gamma$ with $|\alpha|$ entries all equal to $1$
and such that $\beta\setminus\gamma$ is a~horizontal strip, i.e. $\gamma_i\leq\beta_i\leq\gamma_i+1$ for all $i$. By $\mathcal{LR}_1$ we denote the set of all LR-tableaux with entries equal to $1$.
Note that  an~LR-tableau of type $(\alpha,\beta,\gamma)$ and with entries equal to one
is uniquely determined by the partitions $\beta,\gamma$.
We will identify an~LR-tableau $X\in\mathcal{LR}_1$ with the corresponding 
pair $(\gamma^X, \beta^X)$ of partitions. 

By $\emptyset$ we denote the LR tableau such that $\beta^\emptyset=\gamma^\emptyset=(0)$.

\begin{ex}
The following LR-tableau with entries equal to $1$
is uniquely determined by partitions $\beta^X=(7,7,5,2,2,1)$
and $\gamma^X=(7,6,4,1,1,1)$:
$$X=\ytableausetup{centertableaux}
\ytableausetup{smalltableaux}
\ytableaushort
{\none\none\none\none\none\none,\none\none\none11,\none\none\none,\none\none\none,\none\none1,\none\none,\none1}
* {6,5,3,3,3,2,2}
$$
\end{ex}

Below we present algorithmic definition of operations
$$
*\; :\; \mathcal{LR}_1 \times \mathcal{LR}_1 \to \mathcal{LR}_1
$$
and 
$$
*\; :\; (\mathcal{LR}_1\times\mathbb{N}) \times (\mathcal{LR}_1\times\mathbb{N}) \to (\mathcal{LR}_1\times\mathbb{N}).
$$

The following algorithm is given in \cite{kk-generic} and defines the operation 
$*\; :\; \mathcal{LR}_1 \times \mathcal{LR}_1 \to \mathcal{LR}_1
$. 

\begin{alg}\label{algorithm}

{\rm 
{\bf Input.}  $X,Y\in \mathcal{LR}_1$. 

{\bf Output.} $Z=Y*X\in \mathcal{LR}_1$.

\begin{enumerate}

  \item $\gamma^Z=\gamma^X+\gamma^Y$

 \item  set $n_0=0$

\item for any $i=1,\ldots,\min \{\overline{\beta^X_1},\overline{\beta^Y_1}\},$ do
  \begin{enumerate}
   \item put $\beta^Z_i=\beta^X_i+  \gamma^Y_i$  
   \item if $\beta^Y_i\neq \gamma^Y_i$, then put $n_i=n_{i-1}+1$; else put $n_i=n_{i-1}$  
  \end{enumerate}

\item if $\overline{\beta_1^X}>\min \{\overline{\beta^X_1},\overline{\beta^Y_1}\}$, then
for $i=\min \{\overline{\beta^X_1},\overline{\beta^Y_1}\}+1,\ldots,\overline{\beta_1^X}$ put
$$\beta^Z_i=\beta^X_i \mbox{ and } n_i=n_{i-1},$$
  else
for $i=\min \{\overline{\beta^X_1},\overline{\beta^Y_1}\}+1,\ldots,\overline{\beta_1^Y}$
we set
\begin{enumerate}
 \item if ($\gamma_i^Y=\beta_i^Y\,
{\rm and}\,n>0$) then ($\beta_i^Z=\beta_i^Y+1$ and 
 $n_i=n_{i-1}-1$); else ($\beta_i^Z=\beta_i^Y$ and 
 $n_i=n_{i-1}$)  
\end{enumerate}


\item We set \begin{equation*}
\beta^Z=\beta^Z\cup\alpha
\end{equation*}
where $\alpha=(1,1,\ldots,1)$ is a partition with $n_s$ copies of $1$, where $s=\max \{\overline{\beta^X_1},\overline{\beta^Y_1}\}$.
\end{enumerate}
}
\end{alg}

\begin{ex}

Let 
$$X=\ytableausetup{centertableaux}
\ytableaushort
{\none\none\none\none1,\none\none\none\none,\none\none11,\none\none,1}
* {5,4,4,2,1}\quad
Y=\ytableausetup{centertableaux}
\ytableaushort
{\none\none\none\none\none\none,\none\none\none1,\none\none,1}
* {6,4,2,1}$$
Then $\beta^X=(5,4,3,3,1),\gamma^X=(4,4,2,2)$, $\beta^Y=(4,3,2,2,1,1),\gamma^Y=(3,3,2,1,1,1).$
The Step 1 of Algorithm \ref{algorithm} gives 
$$\gamma^Z=\gamma^X+\gamma^Y=(7,7,4,3,1,1).$$
In our case $\min\{\overline{\beta^X_1},\overline{\beta^Y_1}\}=5$ and therefore,
for $i=1,2,3,4,5$, the Step 3 of Algorithm \ref{algorithm} sets $\beta^Z_i=\beta^X_i+\gamma^Y_i$.
We have
$$\beta_1^Z=8,\, \beta_2^Z=7,\, \beta_3^Z=5,\,\beta_4^Z=4,\, \beta_3^Z=2.$$
Note that $n_5=2.$

For $i>5=\overline{\beta^X_1}=\min \{\overline{\beta^X_1},\overline{\beta^Y_1}\}$ the Step 4 of the algorithm gives:
%
%

$$\beta_6^Z=\beta^Y_6+1=2\;{\rm and}\;n_6=n_5-1=1\, ,\;{\rm because}\;\gamma_6^Y=\beta_6^Y.$$

Coefficients of partition $\beta^Z$ for $i=1,2,\ldots,\max \{\overline{\beta^X_1},\overline{\beta^Y_1}\}=s$ are constructed,
but $n_s=1\neq 0$. Therefore the step 5 of Algorithm \ref{algorithm} sets $\beta^Z=\beta^Z\cup (1).$

Finally, we get $\beta^Z=(8,7,5,4,2,2,1)$, $\gamma^Z=(7,7,4,3,1,1),$
and 
$$Z=\ytableausetup{centertableaux}
\ytableaushort
{\none\none\none\none\none\none1,\none\none\none\none11,\none\none\none\none,\none\none\none1,\none\none1,\none\none,\none\none,1}
* {7,6,4,4,3,2,2,1}$$
\end{ex}

The following algorithm is the first part of the definition of 
$
*\; :\; (\mathcal{LR}_1\times\mathbb{N}) \times (\mathcal{LR}_1\times\mathbb{N}) \to (\mathcal{LR}_1\times\mathbb{N}).
$.

\begin{alg}\label{algorithm1}
{\rm 
{\bf Input.}  $X\in \mathcal{LR}_1$. $n\in\mathbb N$.

{\bf Output.}  $(Z,m)=(\emptyset,n)*(X,0)\in \mathcal{LR}_1\times \mathbb{N}$.
\begin{enumerate}
\item let $L=[\;]$ be an empty list
 \item  for any $i=1,\ldots,\overline{\beta}^X_1$ do
\begin{enumerate}
\item if $\beta^X_i=\gamma^X_i$, then add $i$ as the last element of the list $L$
\end{enumerate}
 \item for any $i=\overline{\beta}^X_1,\ldots,1,$ do
\begin{enumerate}
\item if $i\in L$ and $n>0$, then set 
$$\beta_i^Z=\beta_i^X,\,\gamma_i^Z=\gamma_i^Z-1\, {\rm \; and\;}\, n=n-1,$$ 
else set
$$\beta_i^Z=\beta_i^X\,{\rm \; and\;}\,\gamma_i^Z=\gamma_i^Z.$$
\end{enumerate}
\item the result is $(Z,n)$.	
\end{enumerate}}
\end{alg}

\begin{ex}
Let  
$$X=\ytableausetup{centertableaux}
\ytableaushort
{\none\none\none\none\none,\none\none\none\none,\none\none\none1,\none\none,1}
* {5,4,4,2,1}
$$
Note that $\beta^X=(5,4,3,3,1)$ and $\gamma^X=(4,4,3,2,1).$
The list $L$ has a form $L=[2,3,5].$

\begin{enumerate}
\item First we illustrate the algorithm for $n=2.$

Note that $5\in L$ and $n>0$, so we put 
$$\beta_5^Z=\beta_5^X=1,\quad \gamma_5^Z=\gamma_5^Z-1=0$$ and $n:=n-1=1$.
Since $4\not\in L$, we put   
$$\beta_4^Z=\beta_4^X=3\, {\rm\; and\;}\, \gamma_4^Z=\gamma_4^X=2.$$
We have $3\in L$ and $n>0$, then
$$\beta_3^Z=\beta_3^X=3,\quad \gamma_3^Z=\gamma_3^Z-1=2$$ and $n:=n-1=0.$
Since $n=0$,  we have:
$$\beta_2^Z=\beta_2^X=4\, {\rm and}\, \gamma_2^Z=\gamma_2^X=3,$$
$$\beta_1^Z=\beta_1^X=5\, {\rm and}\, \gamma_1^Z=\gamma_1^X=5.$$
Finally 
$\beta^Z=(5,4,3,3,1)$, $\gamma^Z=(4,4,2,2)$, $n=0$ and the result is $(Z,0),$ where
$$Z=\ytableausetup{centertableaux}
\ytableaushort
{\none\none\none\none1,\none\none\none\none,\none\none11,\none\none,1}
* {5,4,4,2,1}
$$

\item Note that applying this procedure to $X$ and  $n=5$ we get the result $(Z,2)$, where
$$Z=\ytableausetup{centertableaux}
\ytableaushort
{\none\none\none\none1,\none\none\none\none,\none\none11,\none1,1}
* {5,4,4,2,1}
$$

\end{enumerate}
\end{ex}

\begin{rem}\label{rem-algorithm1} By an~empty column we mean a~column of arbitrary length that has no entry.
Let $X\in \mathcal{LR}_1$ have $m$ empty columns and let $n\in\mathbb{N}$.
Note that Algorithm \ref{algorithm1} produces the pair $(Z,k)$, where $Z$ is created by  putting $\min\{m,n\}$ elements $1$ in empty columns of $X$ starting from the right most column. If $n$ is bigger than $m$, then $k=n-m$ and otherwise $k=0$.
\end{rem}

Below we give an~algorithmic definition of the operation $
*\; :\; (\mathcal{LR}_1\times\mathbb{N}) \times (\mathcal{LR}_1\times\mathbb{N}) \to (\mathcal{LR}_1\times\mathbb{N}).
$.

\begin{alg}\label{algorithm2}
{\rm
{\bf Input. } $(X,m),(Y,n)\in\mathcal{LR}_1\times\mathbb{N}$

{\bf Output.} $(Z,k)=(Y,n)*(X,m)\in \mathcal{LR}_1\times\mathbb{N}$.
\begin{enumerate}
\item Apply Algorithm \ref{algorithm1} to $X$ and $n$ and get $(T,s)=(\emptyset,n)*(X,0)$
\item Apply Algorithm \ref{algorithm} to $T$ and $Y$ and get $Z=Y* T$.
\item The result is $(Z,s+m)$.
\end{enumerate}
}
\end{alg}

\begin{ex}

Let 
$$X=\ytableausetup{centertableaux}
\ytableaushort
{\none\none\none\none\none,\none\none\none\none,\none\none\none1,\none\none,1}
* {5,4,4,2,1}\quad
Y=\ytableausetup{centertableaux}
\ytableaushort
{\none\none\none\none\none\none,\none\none\none1,\none\none,1}
* {6,4,2,1}$$
and $m=4,\,n=2.$

In the first step we apply Algorithm \ref{algorithm1} to $X$ and $n$ so we obtain 
$$T=\ytableausetup{centertableaux}
\ytableaushort
{\none\none\none\none1,\none\none\none\none,\none\none11,\none\none,1}
* {5,4,4,2,1}$$ 
and $s=0$ (see example to Algorithm \ref{algorithm1}).

Next by Algorithm \ref{algorithm} for $T$ and $Y$ we obtain the LR-tableau $Z$ (see example to Algorithm \ref{algorithm}).

$$Z=\ytableausetup{centertableaux}
\ytableaushort
{\none\none\none\none\none\none1,\none\none\none\none11,\none\none\none\none,\none\none\none1,\none\none1,\none\none,\none\none,1}
* {7,6,4,4,3,2,2,1}.$$

Finally the result is $(Z,4).$
\end{ex}

\begin{ex} Recall that the operation $*$ defined on $\mathcal{LR}_1$ is associative, see \cite[Lemma 6.4]{kk-generic}.  Now we show that the operation 
$*$ defined on $\mathcal{LR}_1\times \mathbb{N}$ is non-associative. 
Let $N=(\emptyset,1),M=(\ytableausetup{centertableaux}
\ytableaushort
{\none} *{1},0)$.
Then applying Algorithm \ref{algorithm2} we obtain:
$$N*M=(\ytableausetup{centertableaux}
\ytableaushort
{1}*{1},0)\quad {\rm and}\quad M*M=(\ytableausetup{centertableaux}
\ytableaushort
{\none,\none}*{1,1},0)$$
and again applying Algorithm \ref{algorithm2}:
$$(N*M)*M=(\ytableausetup{centertableaux}
\ytableaushort
{\none1}*{2},0)\quad {\rm and}\quad N*(M*M)=(\ytableausetup{centertableaux}
\ytableaushort
{\none,1}*{1,1},0).$$

\end{ex}

\section{The category of homomorphisms of nilpotent linear operators} \label{sec-invariant}

Let $k$ be an~arbitrary field. For a partition $\alpha=(\alpha_1\geq\ldots\geq\alpha_n)$
we denote by $N_\alpha=N_\alpha(k)$ {\bf the nilpotent linear operator of type} $\alpha$, 
i.e. the finite dimensional $k[T]$-module
$$N_\alpha(k)=k[T]/(T^{\alpha_1})\oplus\ldots\oplus k[T]/(T^{\alpha_n}),$$
where $k[T]$ is the $k$-algebra of polynomials with one variable $T$ and $\oplus$ is the~direct sum of $k[T]$-modules.

Denote by $\mathcal{H}=\mathcal{H}(k)$
the category of systems $(N_\alpha,N_\beta,f),$
where $\alpha$, $\beta$ are partitions and $f: N_\alpha\to N_\beta$ is
a~$k[T]$-homomorphism.
Let  $M=(N_\alpha,N_\beta,f)$, $M'=(N_{\alpha'},N_{\beta'},f')$ be objects of
$\mathcal{H}$.
A~morphism $\phi:M\to M'$ is a~pair $(\phi_1,\phi_2)$, where
$\phi_1:N_\alpha\to N_{\alpha'}$, $\phi_2:N_\beta\to N_{\beta'}$ are homomorphisms
of $k[T]$-modules such that $f'\phi_1=\phi_2f$.
Denote by $\mathcal S$ or $\mathcal{S}(k)$ the full subcategory of
$\mathcal{H}$ consisting
of all systems $(N_\alpha,N_\beta,f)$,
where $f$ is a~monomorphism.
For a natural number $n$,
we write $\mathcal H_n$ or $\mathcal{H}_n(k)$ for the full subcategory of $\mathcal{H}$
of all systems $(N_\alpha,N_\beta,f)$ such that $\alpha_1\leq n$. In particular, objects of $\mathcal{H}_1$ are such that $N_\alpha$
is a~semisimple $k[T]$-module.
We will denote by $\mathcal{H}_a^b$ the full subcategory of $\mathcal{H}_1$ consisting of all systems $(N_\alpha,N_\beta,f)$ where $|\alpha|=a$ and $|\beta|=b$.
For a natural number $n$,
we write $\mathcal S_n$ or $\mathcal S_n(k)$ for the full subcategory of $\mathcal S$ consisting of objects from $\mathcal{S}\cap \mathcal{H}_n$.

\subsection{Pickets}
\label{section-pickets}

The  category $\mathcal{H}_1(k)$ is of particular interest for us
in this paper. We shortly describe its properties.
An~object $M=(N_\alpha,N_\beta,f_M)\in \mathcal{H}_1$ induces the following short exact sequence:

$$\xymatrix{0\ar[r]&0\ar[rr]&& N_\beta\ar[rr]^-{id_{N_\beta}} &&N_\beta \ar[r] & 0\\
0\ar[r]& \Ker f_M \ar[u]\ar[rr]^-{u}&& N_\alpha \ar[u]^-{f_M}
\ar[rr]^-{p} && \Coker\, u\ar[u]^-{\overline{f_M}}\ar[r]&0},$$
where $u$ is the natural embedding and $p$ is the canonical epimorphism.
Since  $(N_\alpha,N_\beta,f_M)\in \mathcal{H}_1$, the module $N_\alpha$ is semisimple and therefore 
the lower row splits. So we obtain that $M\simeq M'\oplus M''$ where $M'=(\Ker f_M,0,0)$ and $M''=(\Coker\, u, N_\beta,\overline{f_M})\in \mathcal{S}_1$, i.e. $M'\oplus M''=(\Ker f_M\oplus \Coker\, u ,0\oplus N_\beta,0\oplus \overline{f_M})$.


We present a~description of indecomposable objects in $\mathcal{H}_1$. An object $(N_\alpha,N_\beta,f)$ of the category $\mathcal H$ such that $\beta=(m)$ we call a~{\bf picket}.
First we note that any object of the form $M'=(N_{\alpha},0,0)\in\mathcal{H}_1$ is isomorphic to a direct sum of indecomposable objects of the form:
$$P_1^0=(N_{(1)},0,0).$$ 
It is because $\alpha=(1,1\ldots,1)$ and thus $N_\alpha$ is a~semisimple $k[T]$-module.

Each indecomposable object of $\mathcal{S}_1$ is
isomorphic to a~picket that is, it has the form
$$P_0^m=(0,N_{(m)},0)$$
or
$$P_1^m=(N_{(1)},N_{(m)},\iota)$$
where $\iota(1)=T^{(m-1)}$, see \cite{bhw}.
Whenever we want to emphasize the dependence on the field $k$, we will write $P_\ell^m=P_\ell^m(k)$.

Thanks to this classification we can associate with any object $M$ in $\mathcal{S}_1(k)$ the~LR-tableau $\Gamma(M)\in\mathcal{LR}_1$ and with any object $M\simeq M'\oplus (P_1^0)^m$ in $\mathcal{H}_1(k)$, where
 $M'\in\mathcal{S}_1(k)$, the pair $$\widehat{\Gamma}(M)=(\Gamma(M),m)\in\mathcal{LR}_1\times \mathbb{N},$$
 where $\Gamma(M)=\Gamma(M').$
 First we list LR-tableaux associated with pickets in the following table.

\begin{center}
\begin{tabular}{|c|c|c|}\hline
\multicolumn3{|c|}
             {\raisebox{-1ex}[0mm]{\bf LR-tableaux
                  for  the indecomposable objects in $\mathcal S_1$}}
             \\[2ex] \hline
X & $P_0^m$ & $P_1^m$ \\
\hline
 $\Gamma(X)$ &$
m\left\{\ytableausetup{centertableaux}\ytableausetup{smalltableaux}\begin{ytableau}
\\
\none \\
\none[\vdots]\\
\none\\
\\
\end{ytableau} \right.
$ &
$m\left\{\ytableausetup{centertableaux}\ytableausetup{smalltableaux}
\begin{ytableau}
\\
\none \\
\none[\vdots]\\
\none\\
\\
1
\end{ytableau}\right.$
\\ \hline
\end{tabular}
\end{center}

The LR - tableau for a direct sum $M \oplus M^\prime$ has a diagram given by the union $\beta\cup \beta^\prime$
 of the partitions representing the ambient spaces, and in each row the entries
are obtained by lexicographically ordering the entries in the corresponding rows
in the tableaux for $M$ and $M'$, with empty boxes coming first. 

\begin{ex}
The object  $M=P_0^7\oplus P_1^7\oplus P_1^5\oplus P_1^2\oplus P_1^2\oplus P_0^1\oplus P_1^0\oplus P_1^0$
corresponds to the $\widehat{\Gamma}(M)=(\Gamma(M),2)$, where 
$$\Gamma(M)=\ytableausetup{centertableaux}
\ytableaushort
{\none\none\none\none\none\none,\none\none\none11,\none\none\none,\none\none\none,\none\none1,\none\none,\none1}
* {6,5,3,3,3,2,2}
$$
\end{ex}

\begin{rem}
The following observation is an~easy consequence of the combinatorial classification of objects in $\mathcal{H}_1$.
 Let $M,N\in\mathcal{H}_1$ be such that $\widehat{\Gamma}(M)=(\Gamma(M),m)$ and $\widehat{\Gamma}(N)=(\Gamma(N),n)$. Assume that $\beta^{\Gamma(M)}=\beta^{\Gamma(N)}$
 and $\gamma^{\Gamma(N)}\subseteq \gamma^{\Gamma(M)}$. If $P_1^n$ is a~direct summand of $M$ then $P_1^n$ is a~direct summand of $N$.
 \label{rem-direct-summand}
\end{rem}

For each pair $(M,N)$ of indecomposable objects in $\mathcal{H}_1(k)$ we determine
in the table below
the dimension of the $k$-space $\Hom_{\mathcal{H}}(M,N)$ of homomorphisms $M\to N$ in the category $\mathcal{H}$, see \cite[Lemma 4]{sch} and \cite{kos-sch}.

\begin{center}
\begin{tabular}{| c |
                   p{\dimexpr 0.25\linewidth-2\tabcolsep} |
                   p{\dimexpr 0.25\linewidth-2\tabcolsep} |
                   p{\dimexpr 0.25\linewidth-2\tabcolsep} | }\hline
\multicolumn{4}{|c|}{\bf Dimensions of Spaces $\Hom_{\mathcal{H}}(M,N)$}\\
\hline
  $\quad M\quad $ &  $N=P_0^m$ &  $P_1^m$ & $P_1^0$ \\
\hline \hline
$P_0^\ell$ & $\min\{\ell,m\}$ & $\min\{\ell,m\}$ &0\\
\hline
$P_1^\ell$ & $\min\{\ell -1,m\}$   & $\min\{\ell,m\}$ &1\\
\hline
$P_1^0$ &  0  & 0 & 1\\
\hline
\end{tabular}
\end{center}

\subsection{Generic extensions and the hom-order}

Let $M, N \in \mathcal H_1$. An object $U\in \mathcal H$ is an extension of $N$ by $M$ if there exists a~short exact
sequence of the form:
$$
0 \rightarrow  M \rightarrow U \rightarrow N \rightarrow 0.
$$

Note that the subcategory $\mathcal{H}_1\subseteq \mathcal{H}$ is not closed under extensions. In this paper we are interested
only in the extensions  that are objects of $\mathcal{H}_1$.

Since, for fixed integers $a,b$, the category $\mathcal{H}_a^b\subseteq \mathcal{H}_1$ has only finitely many isomorphism classes of indecomposable objects, the results of \cite{bongartz}, \cite{reineke} (see also \cite[Proposition 4.2]{kk-generic}) imply that given $M\in \mathcal H_a^b$ and $N\in \mathcal H_c^d$ there exists the~unique (up to iso) extension $U\in\mathcal{H}_{a+c}^{b+d}$ of
$N$ by $M$ with minimal dimension of its endomorphism ring ${\rm End}_{\mathcal{H}}(U)$. We call $U$ {\bf the generic extension}
of $N$ by $M$ and denote by $U=N*M$. 

On the set of isomorphism classes of objects in $\mathcal{H}_a^b$
one can define the classical degeneration partial order $\degleq$ that has a~geometric nature (see \cite{kk-generic} for details). Algorithmic approach to degeneration order one can find 
for example in \cite{kos-sch, mroz-zwara}. It is known that the generic extension
$U=N*M$ is the  unique $\degleq$-minimal extension of $N$ by $M$, see \cite{reineke, riedtmann}. It is easy to see that the category $\mathcal{H}_a^b$
is equivalent to a~subcategory of the category of modules
 over a~finite dimensional algebra, see \cite{kk-generic}.
Therefore by \cite{zwara}, the degeneration order is equivalent with hom-order $\homleq$,
because the category $\mathcal{H}_a^b$ has only finitely many isomorphism classes of indecomposable objects. It follows that the generic extension
$U=N*M$ is the extension of $N$ by $M$ that is minimal in hom-order.
In the paper instead of degeneration order we will use hom-order for testing whether an~extension is the generic extension. 

\begin{defin}
Fix natural numbers $a,b$. We say that $M,N\in\mathcal{H}_a^b$ are in the {\bf hom-order}, in symbols $M\homleq N,$
if $$[U,M]\leq [U,N]$$ for any object $U$ in $\mathcal{H}(k)$.
Here we write $[M,N]=\dim_k\Hom_{\mathcal{H}}(M,N)$ for $M,N\in \mathcal{H}$.
\end{defin}

By \cite[page 280]{riedtmann} $M\homleq N$ if and only if 
$$[M,U]\leq [N,U]$$ for any object $U$ in $\mathcal{H}(k)$.

\subsection{Properties of pickets}

We collect some elementary properties of pickets of the form $P_1^0$.
We collect some elementary properties of pickets of the form $P_1^0$.

\begin{lem}\label{lem_extzerowy}
Let $M$ be an~object of $\mathcal{S}_1$
and $N$ be an~object of $\mathcal{H}_1$. We have
\begin{enumerate}[{\rm (1)}]
\item $\Hom_{\mathcal{H}} (P_1^0,M)=0,$
\item $\Ext^1_{\mathcal{H}}(M,P_1^0)=0,$ where $\Ext^1_{\mathcal{H}}(M,P_1^0)=0$ is the first group of extensions of $M$ by $P_1^0$.  
\item if $f:M\oplus(P_1^0)^m\to N$ is a~monomorphism, then $(P_1^0)^m$ is a~direct summand of $N$,
\item $N*(M\oplus (P_1^0)^m)=(N*M)\oplus (P_1^0)^m$.
\end{enumerate}
\end{lem}

\begin{proof}
Let $M=(N_\alpha,N_\beta,f)\in\mathcal{S}_1$.

(1)  Any homomorphism  $\phi=(\phi_1,\phi_2)\in\Hom (P_1^0,M)$ induces the following commutative diagram: 
$$\xymatrix{0\ar[r]^{\phi_2} & N_{\beta}\\
N_{(1)}\ar[u]\ar[r]^{\phi_1} & N_\alpha\ar[u]^{f}}$$ 
Therefore $\phi=(\phi_1,\phi_2)=0$, because $f$ is a~monomorphism .


(2) We consider an extension $U$ of $M$ by $P_1^0$:
$$\xymatrix{0\ar[r] & P_1^0\ar[r] & U \ar[r] & M\ar[r] &0}.$$ 
Since $[P_1^0,U]\neq 0$, by (1) the object $P_1^0$ is a~direct summand of $U$.
Finally $U\simeq P_1^0\oplus M$. 

The statement (3) is an~easy consequence of (1).

(4) Suppose that there exists  an extension $U\in\mathcal{H}_1$ of $N$ by $M\oplus(P_1^0)^m$ and such that 
\begin{equation}\label{lem_genzal0}
[U,U]<[(N*M)\oplus (P_1^0)^m,(N*M)\oplus (P_1^0)^m].
\end{equation}
Since there is a~monomorphism $M\oplus(P_1^0)^m\to U$, by (3) we have $U\simeq \overline{U}\oplus (P_1^0)^m$. It is easy to deduce that there is a~short exact sequence:
$$\xymatrix{& 0\ar[r]& M \ar[r] & \overline{U}\ar[r]& N\ar[r]& 0}.$$
Since $N*M$ is the generic extension of $N$ by $M$, we have $[N*M,N*M]\leq [\overline{U},\overline{U}]$
and $N*M\homleq \overline{U}$.
It follows that  $[N*M,(P_1^0)^m]\leq[\overline{U},(P_1^0)^m]$ and $[(P_1^0)^m,N*M]\leq[(P_1^0)^m,\overline{U}]$. Finally we get 
\begin{multline*}\label{lem_gennier}
[(N*M)\oplus(P_1^0)^m,(N*M)\oplus (P_1^0)^m]=\\ [N*M,N*M]+[(P_1^0)^m,N*M]+[N*M,(P_1^0)^m]+[(P_1^0)^m,(P_1^0)^m]\leq \\
[\overline{U},\overline{U}]+
[(P_1^0)^m,\overline{U}]+[\overline{U},(P_1^0)^m]+[(P_1^0)^m,(P_1^0)^m]=\\
[\overline{U}\oplus (P_1^0)^m,\overline{U}\oplus (P_1^0)^m]=[U,U],
\end{multline*}
This contradicts  (\ref{lem_genzal0}). We are done.
\end{proof}

\begin{lem}\label{lem-hom-order}
\begin{enumerate}
\item[{\rm (1)}] Let $U,V\in \mathcal{H}_1$ be such that $U\homleq V$. If $P_1^0$ is a~direct summand of $U$,
then $P_1^0$ is a~direct summand of $V$.
\item[{\rm (2)}] For $m>k$ we have $(P_0^m \oplus P_1^k)\homleq (P_1^m \oplus P_0^k)$.
\item[{\rm (3)}] For any $m>0$ we have $P_1^m \homleq (P_0^m \oplus P_1^0)$.
\end{enumerate}
\end{lem}

\begin{proof}
(1) Assume that $P_1^0$ is a~direct summand of $U$. It follows that $[P_1^0,U]\neq 0$.
Since $U\homleq V$, we have 
$[P_1^0,V]\neq 0$. Therefore $P_1^0$ is a~direct summand of $V$, because $[P_1^0,W]=0$ for every $W\in\mathcal{S}_1$ (Lemma \ref{lem_extzerowy} (1)). 

(2) Let $W\in\mathcal{H}$. Applying the functor $\rm{Hom}_{\mathcal{H}}(W,-)$ to the following short exact sequence:
$$0\to P_0^k\to P^m_0\oplus P_1^k\to P_1^m\to 0$$ 
we obtain the  exact sequence:
$$0\to \rm{Hom}_{\mathcal{H}}(W,P_0^k)\to \rm{Hom}_{\mathcal{H}}(W,P_0^m\oplus P_1^k)\to \rm{Hom}_{\mathcal{H}}(W,P_1^m)\to {\rm Coker}(f)\to 0.$$
Thus $[W,P_0^m\oplus P_1^k]\leq[W,P_0^m\oplus P_1^k]+\dim_k {\rm Coker}(f)=[W,P_1^m]+[W,P_0^k]=  [W,P_1^m\oplus P_0^k]$ and $(P_0^m \oplus P_1^k)\homleq (P_1^m \oplus P_0^k)$.

(3) This follows by the same method as in the proof of (2). We have to apply the functor $\rm{Hom}_{\mathcal{H}}(W,-)$  to the following short exact sequence:
$$0\to P_0^m\to P^m_1\oplus P_0^k\to P_1^0\to 0.$$ 
\end{proof}

\section{The operation $*$ gives generic extensions}\label{sec-generic}

In this section we prove that operations defined in Section \ref{sec-def-alg} are equivalent to the operation of taking the generic extensions.

The following theorem is proved in \cite{kk-generic}.

\begin{thm}\label{thm_stare}
\begin{enumerate}
\item[{\rm (1)}] Let $M,N$ be objects of the category   
$\mathcal{S}_1$. 
The object $U\in\mathcal{S}_1$, where $\Gamma(U)=\Gamma(N)*\Gamma(M)$ is computed by Algorithm {\rm \ref{algorithm}} for $\Gamma(M),\Gamma(N)$, is the generic extension of $N$ by $M$,
i.e. $U=N*M$. 

\item[{\rm (2)}] The operation $*$ defined on $\mathcal{LR}_1\times \mathcal{LR}_1$
is associative.
\end{enumerate}
\end{thm}

We prove several facts that generalize this result.

\begin{prop}\label{prop_extP10}
Let $M\in\mathcal S_1$  and let $n\in \mathbb{N}$. 
Let $(\Gamma(U),u)=(\emptyset,n)*(\Gamma(M),0)$ be computed by Algorithm {\rm \ref{algorithm1}} for $\Gamma(M)$ and $n$.
The object $U\in\mathcal{H}_1$, such that $\widehat{\Gamma}(U)=(\Gamma(U),u)$,   is the generic extension of $(P_1^0)^n$ by $M$, i.e. $U=(P_1^0)^n*M$. 
\end{prop}

\begin{proof}
Let $U$ be the object  such that $\widehat{\Gamma}(U)=(\Gamma(U),u)=(\emptyset,n)*(\Gamma(M),0)$ is constructed by Algorithm \ref{algorithm1}.  Note that $$
0\rightarrow P_0^{c}\rightarrow P_1^{c}\rightarrow P_1^0\rightarrow 0
$$ is a~short exact sequence 
in the category $\mathcal{H}_1$.
Now, applying Remark \ref{rem-algorithm1}, it is easy to see that
 $U$ is an extension of $(P_1^0)^n$ by $M$.

 Let $M=P_{\varepsilon_1}^{k_1}\oplus\ldots \oplus P_{\varepsilon_s}^{k_s}\in \mathcal{S}_1$,
 where $k_1\geq k_2\geq \ldots\geq k_s\geq 1$ and $\varepsilon_i\in\{0,1\}$
 for all $i$. Set 
 $I=\{1,\ldots,s\}$, $I_0=\{i\in I\;\; ;\;\; \varepsilon_i=0\}$ and $I_1=I\setminus I_0$.
 Let $N$ be an arbitrary extension of $(P_1^0)^n$ by $M$ and let  $\widehat{\Gamma}(N)=(\Gamma(N),n').$ Note that $\bigoplus _{i\in I_1}P_1^{k_i}$ is a~direct summand of $N$.
 Indeed, the short exact sequence
 $$ 0\to M\to N\to (P_1^0)^n\to 0$$
 induces the following commutative diagram with exact rows and columns
 $$\xymatrix{ &0 & 0 & & \\
& N_{\gamma^{\Gamma(M)}}\ar[u]\ar[r] & N_{\gamma^{\Gamma(N)}}\ar[u]& &\\
0\ar[r]& N_{\beta^{\Gamma(M)}}\ar[r]\ar[u] & N_{\beta^{\Gamma(N)}}\ar[u]\ar[r] &0\ar[r] &0\\
0\ar[r] & N_{\alpha^{\Gamma(M)}}\ar[r]\ar[u] & N_{\alpha^{\Gamma(N)}}\ar[u]\ar[r] &N_{(1)^n}\ar[u]\ar[r] &0 \\
&0\ar[u] &N_{(1)^{n'}}\ar[u]\ar[r] & N_{(1)^n}\ar[u]}$$
where $n'\leq n$. It follows that $\beta^{\Gamma(M)}=\beta^{\Gamma(N)}$ and, by the Snake Lemma, we get the following exact sequence
$$
0\to N_{(1)^{n'}}\to  N_{(1)^n}\to N_{\gamma^{\Gamma(M)}}\to N_{\gamma^{\Gamma(N)}}\to 0
$$
Applying \cite[(3.1) page 185]{macd} we get $\gamma^{\Gamma(N)}\subseteq\gamma^{\Gamma(M)}$. Therefore
Remark \ref{rem-direct-summand} implies that the  object $\bigoplus _{i\in I_1}P_1^{k_i}$ is a~direct summand of $U$.
Now, applying Remark \ref{rem-algorithm1} and Lemma \ref{lem-hom-order} (b), (c) we 
can deduce that extension computed by Algorithm \ref{algorithm1} is minimal in hom-order among all extensions
of $(P_1^0)^n$ by $M$. We are done.
\end{proof}

The following result is a~simple consequence of Algorithm \ref{algorithm1}.

\begin{lem}
\label{lem_h1tos1}
Let $X,Y\in\mathcal{LR}_1$, $m\in \mathbb{N}$ and let $(\emptyset,m)*(X,0)=(Y,n) \in\mathcal {LR}_1\times \mathbb{N}$ for some $n\in\mathbb{N}$. Then $(Y,0)=(\emptyset,m-n)*(X,0)$.
\end{lem}

\begin{lem}  \label{gen_lemmapull}
Let $M,N\in \mathcal{S}_1$ and $n\in\mathbb{N}$. 
Let $(T,s)=(\emptyset,n)*(\Gamma(M),0)$ be computed by Algorithm {\rm \ref{algorithm1}}
and $Z=\Gamma(N)*T$ be computed by Algorithm {\rm \ref{algorithm}}.  The generic extension of $N\oplus (P_1^0)^n$ by $M$ is
$W$, where $\widehat{\Gamma}(W)=(Z,s)$.
\end{lem}

\begin{proof}
Let $M\in\mathcal S_1$ and $N\in \mathcal S_1$. Consider $N''\simeq N\oplus (P_1^0)^n$. It follows from Lemma \ref{lem_extzerowy} (2) that $N''$ is the unique
extension  of $N$ by $(P_1^0)^n$.  By Proposition \ref{prop_extP10}, the generic extension $V''$ of $(P_1^0)^n$ by $M$ is such that 
$\widehat{\Gamma}(V'')=(\emptyset,n)*(\Gamma(M),0)$. 
 Let $\widehat{\Gamma}(V'')=(\Gamma(V),s)$ for suitable $V\in\mathcal{S}_1$. By Lemma \ref{lem_h1tos1}, we have $(\Gamma(V),0)=(\emptyset,n-s)*(\Gamma(M),0)$. Thus Proposition \ref{prop_extP10} implies that $V$ is the generic extension of $(P_1^0)^{n-s}$ by $M$. Fix a~short exact sequence

$$
\xymatrix{0\ar[r] & M\ar@{->}^f[r]& V  \ar@{->}[r] & (P_1^0)^{n-s} \ar[r] &0}
$$
 
By Lemma \ref{thm_stare}, the generic extension $W$ of $N$ by $V$ is given by 
$\Gamma(W)=\Gamma(N)*\Gamma(V)$. Fix a~short exact sequence

$$
\xymatrix{0\ar[r] & V\ar@{->}^g[r]& W  \ar@{->}[r] & N \ar[r] &0}
$$

 We show that $W\in \mathcal S_1$  is the generic extension of $N'$ by $M$, where $N'\simeq N\oplus (P_1^0)^{n-s}$.  The following diagram, induced by short exact sequences given above, proves that $W$ is an~extension of $N'$ by $M$:
$$\xymatrix{ &0\ar[d]&0\ar[d]&0\ar[d]&\\
0\ar[r] & M\ar[r]^f\ar@{=}[d] & V  \ar[r]\ar[d]^g & (P_1^0)^{n-s} \ar[r]\ar[d]
 &0\\
0\ar[r]& M\ar[r]^{g\circ f}\ar[d] & W \ar[r]\ar[d] & \overline{N} \ar[r]\ar[d] & 0\\
& 0\ar[r] & N \ar@{=}[r]\ar[d] & N \ar[r]\ar[d] & 0 \\
& & 0 & 0 & 
}$$

Note that $\overline{N}\simeq N'$, because
it follows from Lemma \ref{lem_extzerowy} (2) that $N'$ is the unique
extension  of $N$ by $(P_1^0)^{n-s}$. Consequently the last column splits.

Suppose for a~contrary that there exists an~extension $W'\in \mathcal H_1$
of $N'$ by $M$ such that $W\not\simeq W'$ and $[W',W']\leq[W,W]$.
In this case $W'\homleq W$. 
Since $W\in\mathcal S_1$,  by Lemma \ref{lem-hom-order} (1) we have $W'\in\mathcal S_1$.
Consider the diagram 
$$\xymatrix{ &0\ar[d]&0\ar[d]&0\ar[d]&\\
0\ar[r] & M\ar@{-->}[r]\ar@{=}[d] & V'  \ar@{-->}[r]\ar@{-->}[d] & (P_1^0)^{n-s} \ar[r]\ar[d]^{h_2} &0\\
0\ar[r]& M\ar[r]\ar[d] & W' \ar[r]^{h_1}\ar[d] & N' \ar[r]\ar[d] & 0\\
& 0\ar[r] & N \ar@{=}[r]\ar[d] & N \ar[r]\ar[d] & 0 \\
& & 0 & 0 & 
}$$
where $V'$ is the pullback of the morphisms $h_1$ and $h_2$.
Since $W'\in \mathcal{S}_1$, we have $V'\in \mathcal{S}_1$, because $[P_1^0,\mathcal{S}_1]=0$.
Since $V$ is the generic extension of $(P_1^0)^{n-s}$ by $M$ and $V'$ is an extension of $(P_1^0)^{n-s}$ by $M$, we have $V\degleq V'$ and by \cite[Lemma 5.8]{kk-generic} also $W=N*V\homleq  N* V'\homleq W'$. We proved that $W\homleq W'$ and $W'\homleq W$. This implies that $W\simeq W'$, because $\homleq$ is a~partial order.
This contradiction  proves that $W$
is the generic extension of $N'$ by $M$. 

It remains to prove that $ W\oplus (P_1^0)^s$ is the generic extension of $N''$ by $M$. Let $\overline{W}$ be the generic extension of $N''$ by $M$. 
We construct the following pullback diagram
$$\xymatrix{ &0\ar[d]&0\ar[d]&0\ar[d]&\\
0\ar[r] & M\ar@{-->}[r]\ar@{=}[d] & \overline{V}  \ar@{-->}[r]\ar@{-->}[d] & (P_1^0)^{n} \ar[r]\ar[d]^{h_2} &0\\
0\ar[r]& M\ar[r]\ar[d] & \overline{W} \ar[r]^{h_1}\ar[d] & N'' \ar[r]\ar[d] & 0\\
& 0\ar[r] & N \ar@{=}[r]\ar[d] & N \ar[r]\ar[d] & 0 \\
& & 0 & 0 & 
}$$
Since $V''$ is the generic extension of $(P_1^0)^n$ by $M$ and $(P_1^0)^s$ is a~direct summand of $V''$, we conclude that $(P_1^0)^s$ is a~direct summand of $\overline{V}$. It follows that
$(P_1^0)^s$ is a~direct summand of $\overline{W}$, because $[P_1^0,\mathcal{S}_1]=0$. Therefore
$\overline{W}\simeq \overline{W}'\oplus (P_0^1)^s$.
 We have $\overline{W}\homleq W\oplus (P_0^1)^s$ (because $\overline{W}$ is the generic extension of $N''$ by $M$), and it follows that
 $\overline{W}'\homleq W$. Since $W$ is the generic extension of $N'$ by $M$, we get $W\simeq\overline{W}'$ and consequently $\overline{W}\simeq W\oplus(P_1^0)^s$. We are done.
\end{proof}

\begin{lem}\label{lem_gen_S1}
Let $M\in \mathcal{S}_1$, $N\in\mathcal{H}_1$ and $m\in\mathbb{N}$. The generic extension of $N$ by $M\oplus (P_1^0)^m$ is $U\oplus (P_1^0)^m$, where $U$ is the generic extension of $N$ by $M$. Moreover, $\widehat{\Gamma}(U\oplus (P_1^0)^m)=(\Gamma(U),s+m),$ where $(\Gamma(U),s)=\widehat{\Gamma}(N)*(\Gamma(M),0)$.
\end{lem}

\begin{proof}
It is a~direct consequence of Lemma \ref{lem_extzerowy} (4) and Lemma \ref{gen_lemmapull}.
\end{proof}

\begin{thm}\label{thm-main-gen}
Let $M,N\in\mathcal{H}_1$ be such that $\widehat{\Gamma}(M)=(\Gamma(M),m)$ and $\widehat{\Gamma}(N)=(\Gamma(N),n)$.
The generic extension $U=N*M\in \mathcal{H}_1$ is such that $\widehat{\Gamma}(U)=(\Gamma(N),n)*(\Gamma(M),m)$ is computed by Algorithm {\rm \ref{algorithm2}}.  
\end{thm}

\begin{proof}
Let $M= M'\oplus(P_1^0)^m$ and $N= N'\oplus(P_1^0)^n$,
where $M',N'\in\mathcal{S}_1$. By Lemma \ref{lem_gen_S1}
we have 
$$
N*M=
((N'\oplus(P_1^0)^n)*M)\oplus (P_1^0)^m.
$$
Let $V=(N'\oplus(P_1^0)^n)*M$.
By Lemma \ref{gen_lemmapull}, to get $\widehat{\Gamma}(V)$ we have to
apply first Algorithm \ref{algorithm1} to $\Gamma(M)$ and $n$ and get $(T,s)$,
then we have to apply Algorithm \ref{algorithm} to $T$ and $\Gamma(N')$
to get $Z$. Finally, $\widehat{\Gamma}(V)=(Z,s)$ and 
$\widehat{\Gamma}(N*M)=(Z,s+m)=(\Gamma(N),n)*(\Gamma(M),m)$
\end{proof}

\begin{ex} We rewrite the last example of Section \ref{sec-def-alg} in terms of generic extensions. 
Let $N=P_1^0,\,M=P_0^1$ then we have
$$N*M=P_1^1\;\;\;{\rm and}\;\;\;M*M=P_0^2.$$
Moreover, it easy to check that
$$(N*M)*M=P_1^1*P_0^1=P_0^1\oplus P_1^1 \;\;{\rm and}\;\; N*(M*M)=P_1^0*P_0^2=P_1^2.$$
Therefore $*$ is non-associative.
\end{ex}


\begin{thebibliography}{999}

\bibitem{bhw} D.\ Beers, R.\ Hunter, E.\ Walker,
  {\it Finite valuated p-groups,}
  Abelian Group Theory.\ LNM {\bf 1006}, Springer (1983), 471--506.

\bibitem{bongartz} K.\ Bongartz, {\it On degenerations and extensions of finite dimensional
modules}, Adv.\ Math.\ {\bf 121} (1996), 245--287.

\bibitem{dengdu} 
B. Deng and J. Du, {\it Monomial bases for quantum affine $\mathfrak{sl}_n$}, Adv. Math. {\bf 191} (2005), 276-304.

\bibitem{dengdumah}
B. Deng, J. Du and A. Mah, 
{\it Presenting degenerate Ringel-Hall algebras of cyclic quivers},
J. Pure Appl. Algebra {\bf 214} (2010), 1787-1799. 

\bibitem{hubery}
A. Hubery, {\it The composition algebra and composition monoid of the Kronecker quiver}, J. London Math. Soc. {\bf 72} (2005), 137-150.

\bibitem{kk-generic} M.\ Kaniecki and J.\ Kosakowska, {\it Applications of Littlewood-Richardson tableaux to computing generic extension of semisimple invariant subspaces of nilpotent linear operators}, Linear Algebra Appl. {\bf 588} (2020), 134-159.

\bibitem{kos} J.\ Kosakowska {\it Generic extensions of nilpotent $k[T]$-modules, monoids of partitions and constant terms of Hall polynomials}, Coll. Math. {\bf 128} (2012), 253-261.

\bibitem{kos-sch} J.\ Kosakowska and M.\ Schmidmeier, {\it Operations on arc
diagrams and degenerations for invariant subspaces of linear
operators}, Trans. Amer. Math. Soc. {\bf 367} (2015), 5475-5505.

\bibitem{kos-sch2} J.\ Kosakowska and M.\ Schmidmeier, {\it The boundary of the irreducible components for invariant subspace varieties}, Math. Zeit. {\bf 290} (2018), 953-972

\bibitem{macd} I.\ G.\ Macdonald, {\it Symmetric Functions and Hall Polynomials},
Oxford University Press, 1995.

\bibitem{mroz-zwara}
A.\ Mr\'{o}z and G.\ Zwara {\it Combinatorial algorithms for computing degenerations of modules of finite dimension},
Fund. Inf. {\bf 132} (2014), 519-532.

\bibitem{reineke}
M. Reineke, {\it Generic extensions and multiplicative bases of quantum groups at $q = 0$}, Represent. Theory
{\bf 5} (2001), 147-163.


\bibitem{riedtmann} C.\ Riedtmann, {\it Degenerations for representations of quiver with
relations}, Ann.\ Sci.\ Ec.\ Norm.\ Super.\ {\bf 4} (1986), 275--301.



\bibitem{sch} M.\ Schmidmeier,
  {\it Hall polynomials via automorphisms of short exact sequences,}
  Algebr.\ Represent.\ Theory {\bf(15)} (2012), 449--481.



\bibitem{zwara} G. Zwara, {\it Degenerations for modules over representation-finite
biserial algebras}, J. Algebra {\bf 198} (1997), 563-581.


\end{thebibliography}
\end{document}